\documentclass[11pt,a4paper]{amsart}
\usepackage{amsmath,amssymb,amscd}
\usepackage{tikz}
\usetikzlibrary{cd}
\newtheorem{proposition}{Proposition}
\newtheorem{theorem}{Theorem}[section]
\newtheorem{corollary}[theorem]{Corollary}
\newtheorem{lemma}[theorem]{Lemma}
{\theoremstyle{definition}
\newtheorem{definition}{Definition}[section]
 
 }
 {\theoremstyle{remark}
\newtheorem{remark}{Remark}
 
}

\newcommand{\G}{{\operatorname{G}_2}}
\newcommand{\lG}{\check{\operatorname{G}}_2}
\newcommand{\lS}{\check{\operatorname{S}}_2}
\newcommand{\gc}{{h}}
\newcommand{\fs}{{W}}
\newcommand{\cyl}{\mathfrak C}
\newcommand{\bl}{\mathcal I}
\newcommand{\bfs}{{\lG\fs}}
\newcommand{\lF}{\check{F}}
\newcommand{\trace}{\operatorname{trace}}

\newcommand{\Id}{\operatorname{Id}}

\newcommand{\vect}{\operatorname{span}}
\newcommand{\Hol}{\operatorname{Hol}}
\newcommand{\Lip}{\operatorname{Lip}}

\newcommand{\D}{\mathbb D}
\renewcommand{\j}{\mathfrak i}
\def\<{\left\langle}
\def\>{\right\rangle}

\def\dif{\text{ d}} 
\def\<{\left\langle}
\def\>{\right\rangle}
\def\R{\mathbf R}

\def\car{\mathbf 1}

\def\V{V}
\newcommand{\esp}[1]{\mathbf E\left[#1\right]}

\def\P{\mathbf P}

\def\wB{B}
\def\wX{X}
\def\ws{W}
\def\fs{{W}}

\def\I{{\mathcal I}} 
\def\c1var{{C^{1-\text{var}}}}

\def\distKR{\text{dist}_{\text{\textsc{kr}}}}
\def\eB{{\mathbb B}}
\def\cA{{\mathcal A}}

\def\fV{W}
\def\fH{{H}}

\begin{document}
\title{Stein's method for rough paths}
\author{L. Coutin \and L. Decreusefond}
\address{Institut Math\'ematique de Toulouse, Universit\'e
  P. Sabatier, Toulouse, France}
\email{coutin@math.univ-toulouse.fr}
\address{LTCI, Telecom ParisTech, Universit\'e Paris-Saclay, 75013, Paris,
  France}
\email{laurent.decreusefond@telecom-paristech.fr}
\date{}

\begin{abstract}
  The original Donsker theorem says that a standard random walk converges in
  distribution to a Brownian motion in the space of continuous functions. It has
  recently been extended to enriched random walks and enriched Brownian motion.
  We use the Stein-Dirichlet method to precise the rate of this convergence  in the
  topology of fractional Sobolev spaces.
\end{abstract}

\maketitle{}

\section{Introduction}
\label{sec:introduction}

The Donsker theorem says that a random walk
\begin{equation*}
  X^m(t)=\frac{1}{\sqrt{m}}\sum_{k=1}^{[mt]} X_k
\end{equation*}
where the $X_k$'s are independent, identically distributed random variables with
mean $0$ and variance $1$, converges in a functional space to the Brownian
motion $B$. In the original version (see \cite{donsker_invariance_1951}), the
convergence was proved to hold in the space of continuous functions. The first
evolution was the paper of Lamperti \cite{lamperti_convergence_1962}, which
proved the convergence in H\"older spaces. Namely, he stated that if the
increments of the $X_k$'s are $p$-integrable then the random walk $X^m$
converges to $B$ in $\Hol(1/2-1/p)$. The higher the integrability, the stronger
the topology. There are numerous other extensions which can be made to the
Donsker theorem. In the 90s, Barbour \cite{barbour_steins_1990} estimated the
rate of convergence in the space $\mathcal C$ of continuous functions on $[0,1]$
equipped with a stronger topology than the usual sup-norm topology. He proved
that, provided that $\esp{|X_k|^3}$ is finite, then
\begin{equation*}
  \sup_{\|F\|_M\le 1} \esp{F(X^m)}-\esp{F(B)}\le c\, \frac{\log m}{\sqrt{m}},
\end{equation*}
where $M$ is roughly speaking, the set of thrice Fr\'echet differentiable
functions on $\mathcal C$ with bounded derivatives and $\|F\|_M$ is a function of
the supremum of $D^i F,\, i=0,\cdots,3$ over $\mathcal C$. The strategy is to
compare $X^m$ with $B^m$, the affine interpolation, of mesh $1/m$, of the
Brownian motion and then to compare $B^m$ with $B$. The latter comparison is a
sample-path comparison since the two processes live in the same probability
space. This is the part which yields the $\log m$ factor. The former comparison
is done via the Stein's method in finite dimension. The rate of this convergence
is as usual (see~\cite{barbour_introduction_2005,nourdin_normal_2012}) for
Gaussian limits, of the order of $m^{-1/2}$.

In \cite{coutin_steins_2013}, we quantified the rate of convergence of $X^m$
towards $B$ in Besov-Liouville spaces (see \eqref{normedansIap}, which are one
scale of fractional Sobolev spaces. The spaces we considered were not included
in H\"older spaces but the method could  be adapted to obtain convergence
rate in H\"older spaces. However, it would not fit to the present context where
we are considering enriched-paths (see definition below).

Note that in both \cite{barbour_steins_1990} and \cite{coutin_steins_2013}, an
higher integrability of the $X_k$'s would not improve the convergence rates but
would give more flexibility on the choice of the topology in which the
convergence holds: The higher the integrability, the higher the H\"older
exponent may be chosen.

In \cite[Theorem 13.3.3]{friz_multidimensional_2010}, Friz and Victoir
essentially showed that a Lamperti's like result holds for the convergence of
the enriched random walk in the sense of rough path to the enriched Brownian
motion.

The motivation of this paper is to quantify the rate of this convergence in
rough-paths sense. The first difficulty is that the limiting process is no
longer a Gaussian process: The L\'evy area of a Brownian motion is not Gaussian.
Hence we cannot expect to have a direct application of the Stein's method.
However, there is no more randomness in the L\'evy area that there is in the
Brownian motion itself: The L\'evy area is adapted to the filtration generated
by the underlying Brownian motions. Saying that has two consequences. First,
that the probability space we have to consider depends only on the Brownian
motion. Moreover, we have to find functional spaces for which the map which
sends a Brownian motion to its L\'evy area is not only continuous but also
Lipschitz. As mentioned in \cite{friz_variation_2005}, the Besov-Liouville
spaces (we used in \cite{coutin_steins_2013}) are not well fitted to deal with
the iterated integral processes we encounter in rough-paths theory. It is much
better to work with the Slobodetsky scale of fractional Sobolev spaces.

Once the functional framework is set up, in order to avoid some complicated
calculations in infinite dimensional spaces, the idea is to go back to the
approach of \cite{barbour_steins_1990}: Comparing the random walk with the
affine interpolation of the Brownian motion in the Slobodetsky scale of
fractional Sobolev spaces. This can be done by an application of the Stein's
method in finite dimension. The novelty comes from the treatment of the iterated
integrals whose existence hugely complicates the computations of the
remainders. The final result is obtained by considering the known distance
between the enriched affine interpolation and the enriched Brownian motion in
fractional Sobolev spaces.

Rough paths theory is essentially a deterministic theory, it is therefore
tempting to make the estimate we need in a deterministic setting and then to
take the expectation of these bounds. This turns out to be a misleading
approach. For instance, consider a sequence of centered, independent identically
distributed random variables $(X_n,\, n\ge 1)$ and let $S_n=\sum_{i=1}^n X_i$.
If we evaluate the $p$-th moment of $S_n$ with H\"older inequality, we get that
this moment is bounded by a constant times $n^p$. But if we use, as in the
sequel, the Burkholder-Davis-Gundy inequality for discrete time martingale, we
get an upper-bound proportional to $n^{p/2}$. This martingale argument which is
implicitly used in \cite{lamperti_convergence_1962} is the key to our work. For
the sake of simplicity, this implies to separate the treatment of the symmetric
and anti-symmetric parts of the signature. However, this is of no real
importance since, as detailed below, the symmetric part of the signature can be
handled as a classical $\R^d$-valued process.

This paper is organized as follows: In Section \ref{sec:preliminaries}, we give
the necessary notions about fractional Sobolev spaces and rough-paths theory. We
also give a detailed proof of Lamperti's result in the fractional Sobolev spaces
scale for further use and comparison. In Section \ref{sec:rate-convergence-1},
we then define the Kolmogorov-Rubinstein distance and show that this distance
between the random walk and the affine interpolation of the Brownian motion can
be reduced to a problem in finite dimension, should we consider a special set of
Lipschitz functions. In Section~\ref{sec:impr-stein-meth}, we then present our
development of the Stein-Dirichlet method to estimate this distance.

\section{Preliminaries}
\label{sec:preliminaries}

\subsection{Fractional Sobolev spaces}
\label{sec:fract-sobol-spac}
As in \cite{decreusefond_stochastic_2005,friz_multidimensional_2010}, we consider the fractional
Sobolev spaces $\fs_{\eta,p}$ defined for $\eta \in (0,1)$ and $p\ge 1$
as the the closure of ${\mathcal
  C}^1$ functions with respect to the norm
\begin{equation*}
  |f|_{\eta,p}^p=\int_0^1 |f(t)|^p \dif t + \iint_{[0,1]^2}
  \frac{|f(t)-f(s)|^p}{|t-s|^{1+p\eta}}\dif t\dif s.
\end{equation*}
For $\eta=1$, $\fs_{1,p}$ is  the completion of $\mathcal C^1$ for
the norm:
\begin{equation*}
   |f|_{1,p}^p=\int_0^1 |f(t)|^p \dif t + \int_0^1 |f^\prime(t)|^p \dif t.
\end{equation*}
They are known to be Banach spaces and to satisfy the Sobolev embeddings \cite{adams_sobolev_2003,feyel_fractional_1999}:
\begin{equation*}
  \fs_{\eta,p}\subset \Hol(\eta-1/p) \text{ for } \eta-1/p>0
\end{equation*}
and
\begin{equation*}
   \fs_{\eta,p}\subset  \fs_{\gamma,q}\text{ for } 1\ge \eta\ge \gamma  \text{
     and } \eta-1/p\ge \gamma-1/q.
 \end{equation*}
 As a consequence, since $\fs_{1,p}$ is separable (see
 \cite{brezis_analyse_1987}), so does  $\fs_{\eta,p}$. We need to compute the
 $\fs_{\eta,p}$ norm of primitive of step functions.
\begin{lemma}
 \label{lem:normeHDansDual}
 Let $0\le s_{1} < s_{2}\le 1$ and consider
 \begin{equation*}
   h_{s_{1},s_{2}}(t)=\int_{0}^{t} \car_{[s_{1},s_{2}]}(r)\dif r.
 \end{equation*}
 There exists $c>0$ such that for any $s_{1},s_{2}$, we
  have
  \begin{equation}
    \label{eq_donsker:3}
     \|h_{s_{1},s_{2}}\|_{\fs_{\eta,p}} \le c\, |s_{2}-s_{1}|^{1-\eta}.
  \end{equation}
\end{lemma}
\begin{proof}
     Remark that for any $s,t\in [0,1]$,
  \begin{equation*}
    \left| h_{s_{1},s_{2}} (t)-h_{s_{1},s_{2}}(s) \right|\le |t-s|\wedge (s_{2}-s_{1}).
  \end{equation*}
  The result then follows from the definition of the $\fs_{\eta,p}$
    norm.
\end{proof}
 We also need to introduce the Riemann-Liouville fractional spaces for the
 construction of abstract Wiener spaces. For $f\in L^1([0,1];\ dt),$ (denoted by $L^1$ for short) the left
and right fractional integrals of $f$ are defined by~:
        \begin{align*}
          (I_{0^+}^{\gamma}f)(x) & = 
          \frac{1}{\Gamma(\gamma)}\int_0^xf(t)(x-t)^{\gamma-1}\dif t\ ,\ 
          x\ge
          0,\\
          (I_{1^-}^{\gamma}f)(x) & = 
          \frac{1}{\Gamma(\gamma)}\int_x^1f(t)(t-x)^{\gamma-1}\dif t\ ,\ 
          x\le 1,
        \end{align*}
        where $\gamma>0$ and $I^0_{0^+}=I^0_{1^-}=\Id.$ For any
        $\gamma\ge 0$, $p,q\ge 1,$ any $f\in L^p$ and $g\in L^q$ where
        $p^{-1}+q^{-1}\le \gamma$, we have~:
\begin{equation}
  \label{int_parties_frac}
  \int_0^1 f(s)(I_{0^+}^\gamma g)(s)\dif s = \int_0^1 (I_{1^-}^\gamma 
f)(s)g(s)\dif s.
\end{equation}
The Besov-Liouville space $I^\gamma_{0^+}(L^p):= \I_{\gamma,p}^+$ is
usually equipped with the norm~:
\begin{equation}
\label{normedansIap}
  \|  I^{\gamma}_{0^+}f \| _{ \bl_{\gamma,p}^+}=\| f\|_{L^p}.
\end{equation}
Analogously, the Besov-Liouville space $I^\gamma_{1^-}(L^p):= \bl_{\gamma,p}^-$ is
usually equipped with the norm~:
\begin{equation*}
  \| I^{-\gamma}_{1^-}f \| _{ \bl_{\gamma,p}^-}=\|  f\|_{L^p}.
\end{equation*}

It is proved in \cite{Feyel1998} that for $1\ge a>b>c>0$ that  the
following embeddings are continuous (even compact)
\begin{equation*}
  \label{eq:17}
 \fs_{a,p}\subset \bl_{b,p}^+\subset \fs_{c,p}.
\end{equation*}
\subsection{Rough paths}
\label{sec:rough-paths}
We give a quick introduction to the rough-paths theory. For details, we
refer to the monograph \cite{friz_multidimensional_2010}. Consider $T^2(\R^d)$,
the graded algebra of step two:
\begin{equation*}
  T^2(\R^d)=\R \oplus \R^d \oplus (\R^d\otimes\R^d).
\end{equation*}
 We endow $T^2(\R^d)$ with an algebra structure $(+,., \otimes)$
 where for all $(w_0,w_1,w_2)$,  $(z_0,z_1,z_2)\in T^2({\R}^d),~\lambda \in {\R}$
\begin{align*}
(w_0,w_1,w_2)+ (z_0,z_1,z_2)&= (w_0+z_0,\ w_1+z_1,\ w_2+z_2)\\
\lambda.(w_0,w_1,w_2)&= (\lambda w_0,\lambda w_1,\lambda w_2)\\
  (w_0,w_1,w_2)\otimes (z_0,z_1,z_2) &= ( w_0z_0,\  w_0z_1+z_0w_1, \  w_0z_2+z_0w_2 + w_1\otimes z_1).
\end{align*}
Introduce the projection maps: For $i=0,1,2$
\begin{align*}
  \pi_i\, :\, T^2(\R^d)& \longrightarrow (\R^d)^{\otimes i}\\
  (w_0,w_1,w_2)&\longmapsto w_i,
\end{align*}
The set
\begin{multline*}
  1+t^2({\R}^d)=\{w\in T^2(\R^d),\, \pi_0(w)=1\}\\ =\left\{g=(1,w_1,w_2),\ (w_1,w_2) \in {\R}^d \oplus (\R^d\otimes\R^d)\right\}
\end{multline*}
is a Lie group with respect to the tensor multiplication~$\otimes,$
\cite[Prop. 7.17]{friz_multidimensional_2010}. Note that
$$ (1,w_1,w_2)^{-1}= (1,-w_1, -w_2 + w_1\otimes w_1).$$   As usual, a Lie
group, like $ 1+t^2({\R}^d)$, leads to a Lie algebra when equipped with  notions
of product and commutator. Here, the Lie algebra is $(t^2(\R^d),+,.)$ with product
$\otimes$ and commutator
\begin{multline*}
  [g,w]=g\otimes w-w\otimes g=0\ \oplus \ \bigl( \pi_1(g)\otimes\pi_1(w)-\pi_1(w)\otimes\pi_1(g) \bigr),\\ \text{ for any } w,g \in t^2(\R^d).
\end{multline*}
Denote by $(e_i,\, 1\le i\le d)$ the canonical basis of $\R^d$, so
that $(e_i\otimes e_j,\, 1\le i,j\le d)$ is the canonical basis of
$\R^d\otimes \R^d$. Then
\begin{multline*}
  [\sum_{i=1}^d a_i \, e_i+ \sum_{i,j=1}^d c_{i,j}\,  e_i \otimes e_j,\  \sum_{i=1}^d b_i\, e_i+ \sum_{i,j=1}^d f_{i,j} \,e_i \otimes e_j]\\ = \sum_{i<j} (a_i b_j-a_jb_i)\,[e_i,e_j].
\end{multline*}
The exponential and logarithm maps are useful to go  back and forth between
$t^2(\R^d)$ and $1+t^2(\R^d)$:
\begin{align*}
  \exp \, :\, t^2(\R^d)&\longrightarrow 1+t^2(\R^d)\\
  w &\longmapsto 1+w+\frac12 (\pi_1w)^{\otimes 2}
\end{align*}
and
\begin{align*}
  \log \, :\, 1+t^2(\R^d)&\longrightarrow t^2(\R^d)\\
  1+w&\longmapsto w-\frac12 (\pi_1w)^{\otimes 2}.
\end{align*}
We denote by $\Sigma$ the set
of finite partitions $\sigma=\{t_1,\cdots,t_n\}$ of $[0,1]$.
A  continuous path $z$ from $[0,1]$ into $\R^d$ is said to have $1$-finite
variation whenever
\begin{equation*}
  \sup_{\sigma = \{t_1,\cdots,t_n\}\in \Sigma} \sum_{i=1}^{n-1} |z_{t_{i+1}}-z_{t_i}|<\infty.
\end{equation*}
The set of such functions equipped with this quantity as a norm is
denoted by $\c1var$.
 \begin{definition}
The step-2 signature of $z\in \c1var$  is given by:
\begin{align*}
  S_2(z) \, :\, [0,1] &\longrightarrow 1+t^2({\R}^d)\\
t&\longmapsto  \left(1,\  z_t-z_0, \ \int_0^t (z_s-z_0) \otimes dz_s\right).
\end{align*}
The free nilpotent group of order 2, $G^2({\R}^d),$  is the closed subgroup of $1+t^2({\R}^d)$ defined by
$$ G^2({\R}^d)= \left\{S_2(z),\, z \in C^{1-Var}\right\}.  $$
We also consider $\lG({\R}^d)$ ($\lG$ for short since $d$ is fixed), the image of $\G({\R}^d)$ by the logarithm
map. For $z\in \c1var$, this corresponds to consider only the anti-symmetric
part of $\pi_2(S_2(z))$:
\begin{align*}
  \lS(z) \, :\, [0,1] &\longrightarrow t^2({\R}^d)\\
t&\longmapsto  \left(\  z_t-z_0, \ \int_0^t \Bigl[ (z_s-z_0) , \dif z_s \Bigr]\right).
\end{align*}
\end{definition}
\begin{remark}
If
\begin{equation*}
  z(t)=\sum_{i=1}^d\sum_{k=1}^m z_{ik} \,h_k(t)\ e_i
\end{equation*}
where
$(h_1,\cdots,h_m)$ are elements of $\c1var$, we have
\begin{multline}\label{eq_donsker:2}
   \log S_2(z)(t)=\biggl(\  z_t-z_0, \\
     \sum_{1\le i <j\le d}\sum_{1\le k < l\le m} (z_{ik}z_{jl}-z_{il}z_{jk})\left( \int_0^th_k(s)\dif
   h_l(s)-\int_0^t h_l(s)\dif h_k(s) \right) \ [e_i,e_j]\biggr).
\end{multline}
For the sake of notations, we set
$\cA=\{1,\cdots,d\}\times\{1,\cdots,m\}$ and define the $\prec $ relation by:
\begin{equation*}
  a=(a_1,a_2)\prec b=(b_1,b_2) \Longleftrightarrow (a_1<b_1) \text{ and } (a_2<b_2).
\end{equation*}
With these notations, Eqn.~\eqref{eq_donsker:2} then becomes
\begin{multline*}
   \log S_2(z)(t)=\biggl(\  z_t-z_0, \\
     \sum_{a\prec b} [z_{a},\,z_{b}]\left( \int_0^th_{a_2}(s)\dif
   h_{b_2}(s)-\int_0^t h_{b_2}(s)\dif h_{a_2}(s) \right) \ [e_{a_1},e_{b_1}]\biggr).
\end{multline*}

\end{remark}

The group  $G^2({\R}^d)$ has the structure of a sub-Riemannian
manifold. We will not dwell into the meanders of this very rich but
intricate structure. It suffices to say that we can proceed
equivalently by considering usual norms as follows.

For $\alpha \in (0,1)$, a path $w=1\oplus w_1\oplus w_2$  is
said to be $\alpha$-H\"older whenever
\begin{equation*}
\rho_{\alpha}(w)=\max\left(  \sup_{s\neq
    t}\frac{|w_1(t)-w_1(s)|}{|t-s|^\alpha},\
\sup_{s\neq
    t}\frac{|\pi_2(w(s)^{-1}\otimes w(t))|^{1/2}}{|t-s|^{\alpha}}\right)<\infty.
\end{equation*}
Note that for $z\in \c1var$,
\begin{equation*}
  \rho_\alpha(S_2(z))=\max\left(  \sup_{s\neq
    t}\frac{|z(t)-z(s)|}{|t-s|^\alpha},\
\sup_{s\neq
    t}\dfrac{\left|\int_s^t (z_r-z_s)\otimes \dif z_r\right|^{1/2}}{|t-s|^{\alpha}}\right).
\end{equation*}
\begin{definition}
  We denote by $H_\alpha(t^2(\R^d))$, the vector space of paths $w$
  from $[0,1]$ into $t^2(\R^d)$ such that $\rho_\alpha(w)$ is
  finite. It is equipped with the homogeneous norm: For $w$ and $v$ in $H_\alpha(t^2(\R^d))$
  \begin{equation*}
    \|w-v\|_{H_\alpha(t^2(\R^d))}=\rho_\alpha(w-v)=\rho_\alpha\bigl((w_1-v_1)\oplus(w_2-v_2)\bigr).
  \end{equation*}
\end{definition}
Unfortunately, as mentioned in \cite[Chapter 8.3]{friz_multidimensional_2010}, this
metric space is complete but not separable, which  is unacceptable for
our purpose (see Definition~\ref{def_donsker:1} and the remark below).
We thus introduce fractional Sobolev spaces as in \cite{friz_variation_2005}.
\begin{definition}
  For any $\eta\in (0,1)$, any $p\ge 2$, $\bfs_{\eta,p}$
  is the vector space of paths $w$ from $[0,1]$ into $t^2(\R^d)$ such that
  \begin{equation*}
  \|w_1\|^p_{\eta,p}+\iint_{[0,1]^2}\dfrac{\Bigl|\pi_2[w_s^{-1},\,w_t]\Bigr|^{p/2}}{|t-s|^{1+\eta p}}\dif s \dif t<\infty.
  \end{equation*}
  The distance on $\bfs_{\eta,p}$ is defined by
  \begin{multline*}
    \|w-v\|_{\bfs_{\eta,p}} = \|\pi_1(w)-\pi_1(v)\|_{\fs_{\eta,p}}\\
   +\left( \iint_{[0,1]^2}\dfrac{\Bigl|\pi_2[w_s^{-1},\,w_t]-\pi_2[v_s^{-1},\,v_t]\Bigr|^{p/2}}{|t-s|^{1+\eta p}}\dif s \dif t \right)^{1/p}.
  \end{multline*}
\end{definition}
Following \cite{friz_multidimensional_2010}, we know that $\bfs_{\eta,p}$
  is a Banach space included into $H_\alpha(t^2(\R^d))$ for
  $\alpha=\eta-1/p$, provided $\alpha >0$.
  \begin{lemma}
    \label{lem_rough_donsker:separability}
    For any $\eta\in (0,1)$, any $p\ge 2$, $\bfs_{\eta,p}$ is separable.
  \end{lemma}
  \begin{proof}
    Consider the map $\kappa$ defined as
    \begin{align*}
      \kappa\, :\, ([0,1]\rightarrow t_2(\R^d)) & \longrightarrow
                                                    ([0,1]^2\rightarrow
                                                    t_2(\R^d))\\
      w&\longmapsto \biggl((s,t)\mapsto\Bigl(\pi_1(w_t)-\pi_1(w_s),\,\pi_2 \bigl[ w_s^{-1},\, w_t\bigr]\Bigr)\biggr).
    \end{align*}
    Consider the measure $\dif\mu_{\eta,p}(s,t)=|t-s|^{-1-\eta p}\dif s \dif
    t$. Then,
    \begin{equation}\label{eq_rough_donsker:2}
      \|w\|_{\bfs_{\eta,p}}
      =\|\pi_1\circ \kappa(w)\|_{L^p(\mu_{\eta,p})}+\|\pi_2\circ \kappa(w)\|_{L^{p/2}(\mu_{2\eta,p/2})}^{1/2}.
    \end{equation}
    For any $p\ge 1$ and $\eta\in (0,1)$, $L^p(\mu_{\eta,p})$ is
    isometrically isomorphic to $L^p(\text{d}s\otimes  \dif t)$ hence it is
    separable. This entails that  $E=L^p(\mu_{\eta,p})\times
    L^{p/2}(\mu_{2\eta,p/2})$ is separable. Equation
    \eqref{eq_rough_donsker:2} means that the application $T$ which maps $w\in
    \bfs_{\eta,p}$ to the couple $(\pi_1\circ \kappa(w),\ \pi_2\circ
    \kappa(w) )\in E$, is an isometry. Thus $\bfs_{\eta,p}$ is
    isometrically isomorphic to a closed subspace of the separable
    space $E$, hence it is separable.
  \end{proof}

\subsection{Donsker-Lamperti theorem}
\label{sec:donsk-lamp-theor}

For the sake of completeness and for further comparison, we give the
proof of the Donsker-Lamperti theorem in the scale of fractional
Sobolev spaces, which induces the convergence in H\"older spaces.
\begin{definition}
The random walk associated to
the sequence $(X_k,\, k\ge 1)$ is defined by
\begin{equation*}
  X^m(t)=\sqrt{m}\ \sum_{k=1}^m X_k \, r_k^m(t) = \sum_{k=1}^m X_k \, h_{s_1,s_2}(t)
\end{equation*}
where
\begin{equation}
\label{eq_donsker:11}  r_k^m(t)=\int_0^t \car_{((k-1)/m,\ k/m]}(s)\dif s \text{ and }
  h_{s_1,s_2}=\sqrt{m}\ r_k^m.
\end{equation}
\end{definition}
\begin{theorem}
  \label{thm:majo_var}
  If for any $k\ge 1$, $X_k$ belongs to $L^{p}$ for some $p\ge 2$,
  then there exists $c>0$ such that
  \begin{equation}
    \label{eq_donsker_wiener:7} \sup_{m\ge 1} \,    \frac{  \esp{\Bigl|\sum_{k=1}^m X_k \ \sqrt{m}\,\Bigl(r_k^m(t)-r_k^m(s)\Bigr)\Bigr|
         ^{p}}}{|t-s|^{p/2}}<c\, \esp{|X_1|^p }
  \end{equation}
\end{theorem}
\begin{proof}
  For $0\le s < t\le 1$ fixed, the discrete time process
\begin{equation*}
  Y^{st}\, :\, m\longmapsto Y^{st}_m = \sum_{k=1}^m X_k (r_k^m(t)-r_k^m(s))
\end{equation*}
is a martingale with respect to the filtration $\mathcal
F_n=\sigma(X_k,\, 1\le k \le m)$. The Burkholder-Davis-Gundy \cite{rogers_diffusions_2000}
 entails that
\begin{multline*}
m^p   \esp{\left|\sum_{k=1}^m X_k \,\Bigl(r_k^m(t)-r_k^m(s)\Bigr)\right| ^{p}}\\ \le c\, m^{p/2}
  \esp{\left|\sum_{k=1}^m X_k^2 \,\Bigl(r_k^m(t)-r_k^m(s)\Bigr)^2\right| ^{p/2}}.
\end{multline*}
If $|t-s|\le 1/m$, there is at most two values of $k$ such that
$r_k^m(t)-r_k^m(s)$ is not zero. Furthermore,
\begin{equation*}
  |r_k^m(t)-r_k^m(s)| \le |t-s|
\end{equation*}
hence
\begin{equation*}
  |r_k^m(t)-r_k^m(s)|^2\le m^{-1}|t-s|.
\end{equation*}
In this situation,
\begin{equation*}
  \esp{\left|\sum_{k=1}^m X_k^2 \,\Bigl(r_k^m(t)-r_k^m(s)\Bigr)^2\right| ^{p/2}}\le
  c\,m^{-p/2}\, \esp{|X_1|^{p}}\ |t-s|^{p/2},
\end{equation*}
so that \eqref{eq_donsker_wiener:7} holds true for $|t-s|\le 1/m$.
For $|t-s|>1/m$, we remark that $r_k^m(t)-r_k^m(s)$ is not null for at most $[m(t-s)]+2$ values of $k$
and since $r_k^m$ is Lipschitz continuous, $|r_k^m(t)-r_k^m(s)|\le 1/m$ for such
value of $k$. Hence, by convexity inequality,
\begin{multline*}
  \esp{\left|\sum_{k=1}^m X_k^2 \,\Bigl(r_k^m(t)-r_k^m(s)\Bigl)^2\right| ^{p/2}}\le
  c\, (m|t-s|+2)^{p/2}\,m^{-p}\,\esp{|X_1|^{p}}\\
\le c m^{-p/2}\ \esp{|X_1|^{p}} \ |t-s|^{p/2}.
\end{multline*}
Hence, \eqref{eq_donsker_wiener:7} is true for $|t-s|\ge 1/m.$
\end{proof}
It is then straightforward that we have:
\begin{corollary}
  \label{cor:lamperti}
   Assume that for any $k\ge 1$, $X_k$ belongs to $L^{p}$ for some $p\ge 2$.
   Then, for any $\eta<1/2$,
   \begin{equation*}
     \sup_{m\ge 1}\esp{\|X^m\|_{\fs_{\eta,p}}^p}<\infty.
   \end{equation*}
 \end{corollary}
 \begin{proof}
   Actually, $(s,t)\mapsto |t-s|^{p/2}$ is $\mu_{\eta,p}$-integrable provided
   that $p(1/2-\eta)>0,$ i.e. $\eta<1/2.$
 \end{proof}
 \begin{corollary}[Lamperti]
   Assume that for any $k\ge 1$, $X_k$ belongs to $L^{p}$ for some $p\ge 2$.
Then, for any $1/p<\eta<1/2$,  the sequence
$(X^m,\,  m\ge 1)$ converges in distribution in
$\Hol(\eta-1/p)$ to $B$.
\end{corollary}
\begin{proof}
  It is well-known that the finite dimensional distributions of $X^m$
  converge to that of $B$. From Corollary~\ref{cor:lamperti}, we know that for
  any $0<\zeta<1/2$, for any $\epsilon>0$, there
  exists $K_{\epsilon}$ such that
  \begin{equation*}
    \sup_{m\ge 1}\P(\|X^m\|_{\fs_{\zeta,\,p}}\ge K_{\epsilon}) \le \epsilon.
  \end{equation*}
For $0<\eta<\zeta<1/2$, the embedding of $\fs_{\zeta,p}$ into $\fs_{\eta,\,p}$ is
compact: The $\fs_{\zeta,\,p}$-ball of radius $K_{\epsilon}$ is compact in $\fs_{\eta,p}$.
Thus, the sequence $(X^m,\,  m\ge 1)$ is tight in
$\fs_{\eta,p}$ hence convergent. The result follows by the
continuous embedding of $\fs_{\eta,p}$ into $\Hol(\eta-1/p)$.
\end{proof}

\subsection{Abstract Wiener spaces}
\label{sec:gelfand-triplet-1}
The construction of the Gaussian measure on a Banach space is a delicate
question, we refer to \cite{pietsch_operator_1980,MR0388013} for details.
For $H$ a Hilbert space, a cylindrical set is a set of the form
\begin{equation*}
  Z=\{x\in H,\ (\<x,h_1\>_H,\cdots,\<x,h_n\>)\in B\} 
\end{equation*}
for some integer $n$, where $(h_1,\cdots,h_n)$ is an orthonormal family of $H$ and $B$ a
Borelean subset of~$\R^n$. Let $m_n$ be the standard Gaussian
measure on~$\R^n$. By setting, $\mu_H(Z)=m_n(B)$, we get a
cylindrical standard Gaussian measure on $H$. To get a Radon measure on a Banach
space $W$, the usual way is to  find a map from $H$ to $W$ which is
radonifying: It is a linear map which transforms a cylindrical measure into a true
regular Radon measure. We will not dwell into the details of this
theory, it suffices to say that we have the following result : (see
\cite[Proposition XV,4,1]{MR0388013} or \cite[25.6.3]{pietsch_operator_1980}), 
\begin{lemma}
  \label{lem_donsker_final:radon}
  Let
  \begin{equation*}
    \Lambda=\{(\eta,p)\in \R^{+}\times\R^{+}, 0< \eta-1/p<1/2\}.
  \end{equation*}
For any $(\eta,p)\in \Lambda$, 
the embedding 
\begin{equation*}
   \bl_{1, 2}\xrightarrow{\quad \iota_{\eta,p}\quad}
   \fs_{\eta,\, p}\subset \Hol(\eta-1/p).
\end{equation*}
is a radonifying map.
\end{lemma}
\begin{definition}
  An abstract Wiener space is a triple $(\iota,H,W)$ where $H$ is a
  separable Hilbert space, $W$ a Banach space and $\iota$ the embedding
  from $H$ into $B$ which has to be to be radonifying.
\end{definition}
We have the following
diagram:
\begin{equation}\label{eq_donsker:14}
  W^* \xrightarrow{\quad \iota^*\quad } H^* \xrightarrow{\quad j_{H^{*},H}\quad } H \xrightarrow{\quad \iota
    \quad} W,
\end{equation}
where $j_{H^{*},H}$ is the bijective isometry between the Hilbert space $H^{*}$
and its dual $H$.
As a direct consequence of Lemma~[\ref{lem_donsker_final:radon}], we have 
\begin{theorem}
  \label{thm_donsker_final:1}
The triple $(\iota_{\eta,p},\ \bl_{1,2},\
\fs_{\eta,\,p})$ is an abstract Wiener space, for any $(\eta,p)\in
\Lambda$.
\end{theorem}
The Wiener measure $\P_{\eta,p}$ on $\fs_{\eta,p}$, 
is defined by its characteristic function: For all $\eta \in \fs_{\eta,p}^*$,
\begin{equation*}
  \int_{\fs_{\eta,p}} e^{i\, \<\eta,\, y\>_{\fs_{\eta,p}^*,\fs_{\eta,p}} }\dif \P_{\eta,p}(y)=\exp(-\frac12\| j_{H^{*},H}\circ\iota_{\eta,p}^* (\eta)\|_{\bl_{1,2}}^2).
\end{equation*}
This means that for any $\eta\in \fs_{\eta,p}^*$, the random variable $
\<\eta,\, y\>_{\fs_{\eta,p}^*,\fs_{\eta,p}}$ is a centered Gaussian random
variable with variance given by
\begin{equation*}
  \int_{\fs_{\eta,p}} \<\eta,\, w\>_{\fs_{\eta,p}^*,\fs_{\eta,p}}^2 \dif \P_{\eta,p}(y)=\| j_{H^{*},H}\circ \iota_{\eta,p}^*(\eta)\|_{\bl_{1,2}}^2.
\end{equation*}
\begin{remark}
  In what follows, as it is customary, we  identify $\bl_{1,2}$ and its dual
  so that the diagram \eqref{eq_donsker:14} becomes
  \begin{equation*}
  \fs_{\eta,p}^* \xrightarrow{\quad \j_{\eta,p}^{*}=j_{H^{*},H}\circ \,\iota^*_{\eta,p}\quad } \bl_{1,2}  \xrightarrow{\quad \iota_{\eta,p}
    \quad} \fs_{\eta,p}.
\end{equation*}
\end{remark}

By construction, $\j_{\eta,p}^*( \fs_{\eta,p}^*)$ is dense in $\bl_{1,2}$ so that we can define
the Wiener integral as follows.
\begin{definition}[Wiener integral]
  The Wiener integral, denoted as $\delta_{\eta,p}$, is the isometric extension
  of the map
  \begin{align*}
   \delta_{\eta,p}\, :\,  \j_{\eta,p}^*(\fs_{\eta,p}^*)\subset\bl_{1,2}&\longrightarrow  L^2(\P_{\eta,p})\\
   \j_{\eta,p}^*(\eta)  &\longmapsto \<\eta,\, y\>_{\fs_{\eta,p}^*,\fs_{\eta,p}}. 
  \end{align*}
\end{definition}
This means that if $h=\lim_{n\to \infty } \j_{\eta,p}^{*}(\eta_{n})$ in
$\bl_{1,2}$,
\begin{equation*}
  \delta_{\eta,p} h(y)=\lim_{n\to \infty} \<\eta_{n},\, y\>_{\fs_{\eta,p}^*,\fs_{\eta,p}} \text{ in }L^{2}(\P_{\eta,p}).
\end{equation*}
\begin{remark}
  As the Dirac measure at point $t\in [0,1]$ belongs to any $\fs_{\eta,p}^{*}$,
  we an search for $\j_{\eta,p}^{*}(\varepsilon_{t})$. For any $h\in
  \bl_{1,2}\subset \fs_{\eta,p}$,
  we must have
  \begin{equation*}
    h(t)=\<\varepsilon_{t},\, h\>_{\fs_{\eta,p}^{*},\fs_{\eta,p}}=\int_{0}^{1}\dot{\overline{\j_{\eta,p}^{*}(\varepsilon_{t})}}(s)\dot{h}(s)\dif s
  \end{equation*}
  hence
  \begin{equation*}
    \dot{\overline{\j_{\eta,p}^{*}(\varepsilon_{t})}}(s)=\car_{[0,t]}(s) \text{ and } \j_{\eta,p}^{*}(\varepsilon_{t})(s)=t\wedge s.
\end{equation*}
This means, that whatever the functional space $\fs_{\eta,p}$ we are
considering,
\begin{equation*}
B_{\eta,p}= (\delta_{\eta,p}(t\wedge .), \, t\in [0,1])
\end{equation*}
is a centered Gaussian process of covariance kernel
\begin{equation*}
  \esp{B_{\eta,p}(t)B_{\eta,p}(s)}=\<t\wedge., \, s\wedge .\>_{\bl_{1,2}}=t\wedge s.
\end{equation*}
Hence, $B_{\eta,p}$ is a standard Brownian motion. Since we work with a sequence
of increasing (in the sense of inclusion) spaces, we remove the subscripts when
no risk of confusion may happen.
\end{remark}

\begin{definition}
  \label{def_donsker_final:1}
A function $F$ from $\ws_{\eta,p}$ into $\R$ is Lipschitz
whenever for any $x$ and $y$ in $\ws_{\eta,p}$,
\begin{equation*}
  |F(x)-F(y)|\le \Vert x-y\Vert_{\ws_{\eta,p}}.
\end{equation*}
The set of such functions is denoted by $\Lip(d_{\fs_{\eta,p}})$.
\end{definition}
\begin{definition}[Ornstein-Uhlenbeck semi-group]
  For any bounded function on $\ws_{\eta,p}$, for any $\tau \ge 0$,
  \begin{equation*}
    P_\tau F(x)=\int_{\ws_{\eta,p}} F(e^{-\tau}x+\beta_\tau y)\dif\P_{\eta,p}(y)
  \end{equation*}
where $\beta_\tau=\sqrt{1-e^{-2\tau}}$.
\end{definition}
The dominated convergence theorem entails that $P_\tau$ is ergodic:
For any $x\in \ws_{\eta,p}$,
\begin{equation*}
  P_\tau F(x)\xrightarrow{\tau\to\infty} \int_{\ws_{\eta,p}} F\dif\P_{\eta,p}.
\end{equation*}
Moreover, the invariance by rotation of Gaussian measures implies that
\begin{equation*}
  \int_{\ws_{\eta,p}} P_\tau F(x)\dif\P_{\eta,p}(x)=
  \int_{\ws_{\eta,p}}F\dif\P_{\eta,p}\text{, for any $\tau \ge 0$.}
\end{equation*}
Otherwise stated, the Gaussian measure on $\fs_{\eta,p}$ is the invariant and stationary
measure of the semi-group $P=(P_\tau, \, \tau \ge 0)$. For details on the
Malliavin gradient, we refer to \cite{Nualart1995b,Ustunel2010}.
\begin{definition}
  \label{def_donsker_final:2} Let $X$ be a Banach space. A function $F\, :\, \ws_{\eta,p}\to
  X$ is said to be cylindrical if it is of the form
  \begin{equation*}
    F(y)=\sum_{j=1}^k f_j(\delta h_1(y),\cdots,\delta h_k(y))\, x_j
  \end{equation*}
where for any $j\in \{1,\cdots,k\}$, $f_j$ belongs to the Schwartz space on $\R^k$, $(h_1,\cdots, h_k)$ are
elements of $\bl_{1,2}$ and $(x_1,\cdots,x_j)$ belong to $X$. The set of such
functions is denoted by $\cyl(X)$.

For $h\in \bl_{1,2}$,
\begin{equation*}
  \< \nabla F, \, h\>_{ \bl_{1,2}}=\sum_{j=1}^k\sum_{l=1}^{k} \partial_lf(\delta h_1(y),\cdots,\delta h_k(y))\, \<h_l,\, h\>_{ \bl_{1,2}}\ x_j,
\end{equation*}
which is equivalent to say
\begin{equation*}
  \nabla F = \sum_{j,l=1}^k \partial_jf(\delta h_1(y),\cdots,\delta h_k(y))\, h_l\otimes\ x_j.
\end{equation*}
The space $\D_{1,2}(X)$ is the closure of the space of cylindrical functions with
respect to the norm
\begin{equation*}
  \|F\|_{1,2}^2=\|F\|_{L^2(\P_{\eta,p};X)}^2+\|\nabla F\|_{L^2(\P_{\eta,p};\bl_{1,2}\otimes X)}^2.
\end{equation*}
By induction, higher order gradients are defined similarly. For any $k\ge 1$,
the norm on the space $\D_{k,2}(X)$ is given by 
\begin{equation*}
   \|F\|_{k,2}^2=\|F\|_{L^2(\P_{\eta,p};X)}^2+\sum_{j=1}^{k}   \|\nabla^{(j)} F\|_{L^2(\P_{\eta,p};\bl_{1,2}^{\otimes j}\otimes X)}^2.
\end{equation*}
\end{definition}
According to \cite{shih_steins_2011}, we have the following properties   of $\P_{t}F$.
\begin{proposition}
  \label{prop:shih}
Let $F\in L^{1}(\fs_{\eta,p},\P_{\eta,p})$ and $x,\, y \in \fs_{\eta,p}$.
For any $t >0$, $P_t F(x)$ belongs to $\D_{k,2}$ for any $k\ge 1$.
Moreover, the operator $\nabla^{(2)}P_t F(x)$
  is trace-class.
  Let $L$ be the formal operator
  defined by 
  \begin{equation*}
    LG(x)=-\< x,\nabla G(x)\>_{\ws_{\eta,p},\ws_{\eta,p}^*} +\trace_{W_{1,2}}(\nabla^{(2)}G(x)).
  \end{equation*}
Then, for any $t >0$, $P_t F$ belongs to the domain of $L$ and
\begin{equation*}
  \frac{d}{dt}P_t F(x)=LP_t F(x).
\end{equation*}
\end{proposition}

\section{Rate of convergence}
\label{sec:rate-convergence-1}
\subsection{Kolmogorov-Rubinstein distance}
\label{sec:kolm-rubinst-dist}
In \cite{lamperti_convergence_1962}, the proof of Lamperti's Theorem is given for one dimensional processes but it
can be straightforwardly adapted to $\R^d$-valued random walks and Brownian
motion: $X^\cA$ becomes the $\R^d$-valued process
\begin{equation*}
  X^\cA(t)=\sqrt{m}\, \sum_{a\in \cA} X_a \ r_{a_2}(t)\,e_{a_1}=\sum_{a\in \cA} X_a \, h_a(t)
\end{equation*}
where $h_a(t) =\sqrt{m}\  r_{a_2}(t)\, e_{a_1}$ and $(X_a,\, a\in \cA)$ is a family of independent identically distributed
random variables of mean $0$ and variance $1$. Furthermore, $B$ is the $d$-dimensional Brownian motion:
\begin{equation*}
  B(t)=\sum_{i=1}^d B_{i}(t) \, e_{i}.
\end{equation*}
The enriched Brownian motion $\eB$, is the $\G(\R^d)$-value process defined by
\begin{equation*}
  \eB(t)=1\oplus B(t)\oplus \sum_{i,j=1}^d \int_0^t B_i(s)\circ \dif B_j(s)\
  e_i\otimes e_j,\text{ for any } t\in [0,1],
\end{equation*}
where the stochastic integrals are to be understood in the
Stratonovitch sense. Theorem 13.32 of \cite{friz_multidimensional_2010} says that $S_2(X^\cA)$
converges to $\eB$ in some H\"older type spaces. Our primary goal is to give the
rate of this convergence. For, we need to define a distance between probability
measures over H\"older spaces. There are several possibilities of such a
definition, the best suited for an estimate by the Stein method  is the
Kolmogorov-Rubinstein\footnote{We stick to the denomination suggested in
  \cite{Villani2003} even if this distance is often called the Wasserstein distance.} distance:
\begin{definition}[Kolmogorov-Rubinstein distance]
\label{def_donsker:1}
  For $\mu$ and $\nu$ two probability measures on a metric space $(W,d_W)$, their Kolmogorov-Rubinstein
  distance is given by
  \begin{equation*}
    \distKR(\mu,\, \nu)=\sup_{F\in \Lip(d_W)}\int_W F\dif \mu-\int_W
    F\dif \nu,
  \end{equation*}
  where
  \begin{equation*}
    \Lip(d_W)=\Bigl\{F\, :\, W\to \R, \ \forall w,\, v\in W, \
    |F(w)-F(v)|\le d_W(w,\, v)\Bigr\}.
  \end{equation*}
\end{definition}
Theorem 11.3.3 of \cite{dudley_real_2002} states that the topology induced
by this distance on the set of probability measures on $W$ is the same
as the topology of the convergence in law whenever the metric
space $W$ is separable. Unfortunately, as we already mentioned, H\"older  spaces
are not separable, thus to have a meaningful result, we turn to work on  fractional Sobolev
spaces. It is of no importance since Sobolev embeddings ensure that convergence
in fractional Sobolev spaces induces convergence in H\"older spaces.

Our new goal is then to estimate the Kolmogorov-Rubinstein distance in $\G\fs_{\eta,p}$
between $\eB$ and $S_2(X^\cA)$. Remark that
\begin{multline*}
 \pi_2 S_2(X^\cA)=\pi_2\lS(X^\cA) \\ +2\sum_{a\in \cA}(X_a^2-1) \int_0^t h_a^m(s)\otimes \dif h_a^m (s)+\sum_{a\in \cA}h_a^m(t)\otimes h_a^m(t)\\=U_1^m+U_2^m+U_3^m.
\end{multline*}
On the other hand,
\begin{multline*}
   \eB(t)=\sum_{1\le i<j\le d}\left( \int_0^t B_i(s)\circ \dif B_j(s) -\int_0^t B_j(s)\circ \dif B_i(s)\right)\
  [e_i,\,  e_j]\\
  + 2 \sum_{i=1}^d \int_0^t B_i(s)\dif B_i(s) \ e_i\otimes e_i+ \sum_{i=1}^d
  t\,e_i\\ =\log \eB(t)+U_2+U_3.
\end{multline*}
where in $U_2$, the stochastic integral is taken in the It\^o sense.
Direct computations show that  $U_3^m-U_3$ tends to $0$ as $1/m$. The
sequence $U_2^m$ converges to $U_2$ as fast as  $X^\cA$ converges 
to $B$; a rate which is expected  and which will turn out to be much slower than $1/m$. In summary, the Kolmogorov-Rubinstein distance between $\eB$ and
$S_2(X^\cA)$ has the same asymptotic behavior as the distance between $\log\eB$
and $\lS(X^\cA)$. Our final objective is then to estimate the
Kolmogorov-Rubinstein distance between the distributions of $\log\eB$
and $\lS(X^\cA)$ in $\lG\fs_{\eta,p}.$
\subsection{Reduction to finite dimension}
\label{sec:reduction-fo-finite}

Should we
follow the same procedure as the one we used in \cite{coutin_steins_2013}, we
would face the same complications to compute the trace term in some infinite
dimensional space. We remark that $X^\cA$ belongs to the finite dimensional space
\begin{equation*}
  \V=\vect\{h_a^m ,\, a\in \cA\}\subset \bl_{1,2}.  
\end{equation*}
Consider $(g_n,\, n\ge 1)$ a complete orthonormal basis of $\V^\perp$ in $\bl_{1,2}$. The It\^o-Nisio Theorem says that $B$ can be represented as the $\fs_{\eta,p}$-convergent sum
\begin{equation*}
  B=\sum_{a\in \cA} \delta h_a^m \, h_a^m + \sum_{n=1}^\infty \delta g_n \ g_n=B_V+B_V^\perp,
\end{equation*}
where $\delta h$ is the Malliavin divergence (or Wiener integral) associated to $B$ (see \cite{Nualart1995b}).
It turns out that $B_V$ is also  the affine interpolation of
$B$ so that  we may use the results of \cite{friz_multidimensional_2010} to estimate the distance
between $\lS(B_V)$ and $\log\eB$. We can  then resort to the Stein's method in finite
dimension to estimate only the distance between $\lS(X^\cA)$ and~$\lS(B_V)$.
We can always write
\begin{multline*}
  \sup_{F\in \Lip(d_{\bfs_{\eta,p}})} \esp{F(\log\eB)}-\esp{F(\lS(X^\cA))}
\\\shoveleft{ \le
                                                                  \sup_{F\in
                                                                  \Lip(d_{\bfs_{\eta,p}})}
                                                                \esp{F(\log\eB)}-\esp{F(\lS(B_V))}}
 \\+\sup_{F\in
                                                                  \Lip(d_{\bfs_{\eta,p}})}
                                                                  \esp{F(\lS(B_V))}-\esp{F(\lS(X^\cA))}
.
\end{multline*}
On the one hand, since $\log\eB$ and $\lS(B_V)$ live on the same probability
space, for $F\in \Lip(d_{\bfs_{\eta,p}})$, according to \cite[Proposition 13.20]{friz_multidimensional_2010},
\begin{multline}\label{eq_rough_donsker:3}
  \esp{F(\log\eB)}-\esp{F(\lS(B_V))}\le
  \esp{\|\log\eB-\lS(B_V)\|_{\bfs_{\eta,p}}}\\ \le
  c\, m^{-(1/2-\eta)}.
\end{multline}
It remains to  estimate
\begin{equation*}
   \sup_{F\in\Lip(d_{\bfs_{\eta,p}})}\esp{F(\lS(B_V))}-\esp{F(\lS(X^\cA))}.
 \end{equation*}
Actually, for technical reasons, we could not make this estimate for $F$ only
Lipschitz. As for the multivariate Gaussian approximation, we must have a
condition on the regularity of the second derivative of test functions.
\begin{definition}
  Let
  \begin{math}
    \bl_{2,2}^{\pm}=( I_{0^{+}}^{1}\circ I_{1^{-}}^{1})(L^{2}).
  \end{math}
  For $F\, :\,\bfs_{\eta,p} \longrightarrow \R $, let $\lF=F\circ \lS$ as
  described in Figure~\ref{fig:lf}.

  \begin{figure}[!ht]
    \centering
      \begin{tikzcd}
      \fs_{\eta,p} \arrow[rd,"\lF" below] \arrow[r,  "\lS"] & \bfs_{\eta,p} \arrow[d, "F"] \\
      & \mathbf R
  \end{tikzcd}
    \caption{Definition of $\lF$.}
    \label{fig:lf}
  \end{figure}
  
  Let
  $\Sigma_{\eta,p}$ be the set of functions $F\in \Lip(\bfs_{\eta,p})$ such that $\nabla^{(2)}\lF$ belongs to $L^{2}(\fs_{\eta,p};\,
  \bl_{2,2}^{\pm}\otimes\bl_{2,2}^{\pm})$  and
  \begin{multline*}
    \left| \<(\nabla^{(2)}\lF)(x)-(\nabla^{(2)}\lF)(x+g),\, h\otimes k\>_{\bl_{1,2}^{\otimes 2}} \right| \\
\le \|h\|_{L^{2}}\,\|k\|_{L^{2}} \  \|\lS(x)-\lS(x+g)\|_{\bfs_{\eta,p}},
\end{multline*}
 for any $x\in \fs_{\eta,p}$,
  for any $g\in \bl_{1,2}$, for any $h,k\in L^{2}$.
\end{definition}
\begin{remark}
  If $\lF$ is thrice differentiable in the direction of $\bl_{1,2}$ with for any
  $x\in \fs_{\eta,p}$, 
  $\nabla^{(j)}\lF(x)\in (\bl_{2,2}^{\pm})^{\otimes j}$ for any $j=1,2,3$ and
  \begin{equation*}
   \|\nabla^{(3)}F\|_{L^{\infty}(W;(\bl_{2,2}^{\pm})^{\otimes 3})}<\infty
 \end{equation*}
 then the fundamental theorem of calculus entails that
  \begin{multline*}
    \left| \<(\nabla^{(2)}\lF)(x)-(\nabla^{(2)}\lF)(x+g),\, h\otimes k\>_{\bl_{1,2}^{\otimes 2}} \right| \\
\le \|\nabla^{(3)}F\|_{L^{\infty}(W;(\bl_{2,2}^{\pm})^{\otimes 3})} \|h\|_{L^{2}}\,\|k\|_{L^{2}} \  \|g\|_{L^{2}},
\end{multline*}
so that $\|\nabla^{(3)}F\|_{L^{\infty}(W;(\bl_{2,2}^{\pm})^{\otimes
    3})}^{-1}\lF$ belongs to $\Sigma_{\eta,p}$.
\end{remark}
 Section~\ref{sec:impr-stein-meth} is devoted to prove our main theorem:
 \begin{theorem}
   \label{thm:final}
   Let $(\eta,p)\in \Lambda$ and $p\ge 3$. If $X_a$ belongs to $L^p$, then 
   \begin{equation}
\label{eq_donsker:16}       \sup_{F\in \Sigma_{\eta,p}}\esp{F(\lS(B_V))}-\esp{F(\lS(X^\cA))}\le  c\,   \|X_a\|_{L^{p}}^{3}\,  m^{-(1/2-\eta)}.
   \end{equation}
 \end{theorem}
 \begin{remark}
   Note that the integrability of the $X_a$'s does enlarge the spaces in
 which the convergence holds, as in the Lamperti Theorem, but it does not modify the rate of convergence.   
\end{remark}
We now detail the proofs of the main estimates.
\section{Stein method}
\label{sec:impr-stein-meth}

We have to estimate
\begin{equation*}
  \sup_{F\in{\Sigma}_{\eta,p}}\esp{F(\lS(B_V))}-\esp{F(\lS(X^\cA))}.
\end{equation*}
Recall that
\begin{equation*}
  X^\cA=\sum_{a\in \cA}X_a h_a^m,\   B_\V=\sum_{a\in \cA}\delta h_a^m \  h_a^m \text{ and } V=\vect\{h_a^m , \, a\in \cA\}.
\end{equation*}
Let $\fV$ (respectively $\fH$) be
the space $V$ equipped with the norm of $\fs_{\eta,p}(\R^d)$ (respectively of
$\bl_{1,2}(\R^d)$). Since $V$ is finite dimensional, the difference between
$\fV$ and $\fH$ is tenuous but still of some importance to clarify the
situation. We have the following situation
\begin{center}
  \begin{tikzcd}
    \fs_{\eta,p}^*\subset W^*
    \arrow[r,"\j_{\eta,p}^*"]  & \fH \arrow[d,"\iota_{\eta,p}"] \subset \bl_{1,2} &&\\
    & W\subset \fs_{\eta,p} \arrow[r,"\lS"] & \bfs_{\eta,p} \arrow[r,"F"] & \R.
  \end{tikzcd}
\end{center}
For the sake of notations, we set $\lF=F\circ \lS$.
The Stein-Dirichlet representation formula (see \cite{Decreusefond2015c}) then
stands that
\begin{equation}\label{eq_donsker_wiener:2}
  \esp{\lF(\wB_V)}-\esp{\lF(\wX^\cA)}=\esp{ \int_0^\infty
    LP_t \lF(\wX^\cA)\dif t }.
  \end{equation}

\subsection{Technical lemmas}
\label{sec:technical}

\emph{In what follows, $c$ is a constant which may vary from line to line. We
  denote $\P_{\eta,p}^{V}$ the restriction of $\P_{\eta,p}$ to $V$: It is the
  distribution of $B_{V}$.}

With the notations of Proposition~\ref{prop:shih}, we
have: 
\begin{lemma}
  \label{lem:transformationL}
  For any $F\in \Lip(d_{\fs_{\eta,p}})$, for any $t>0$,
  \begin{multline*}
    \esp{LP_tF(X^\cA)}\\
    \shoveleft{ =\sum_{a\in \cA}\int_0^1\mathbf E\left[ X_a^2\,\left\langle \nabla^{(2)}P_tF(X^\cA_{\neg
          a})-\nabla^{(2)}P_tF(X^\cA_{\neg a}+rX_ah_a^m),\right.\right.}\\
  \shoveright{        \left.\left. h_a^m\otimes h_a^m
      \vphantom{ \nabla^{(2)}P_tF(X^\cA X_a^2}\right\rangle_{\bl_{1,2}}\right]\dif r}\\
    +\sum_{a\in \cA}\esp{\left\langle
        \nabla^{(2)}P_tF(X^\cA)-\nabla^{(2)}P_tF(X^\cA_{\neg a}), h_a^m\otimes
        h_a^m \right\rangle_{\bl_{1,2}}}.
  \end{multline*}
\end{lemma}
\begin{proof}
  Recall that
  \begin{multline*}
    LP_tF(X^\cA)=-\< X^\cA, \nabla P_tF(X^\cA)\>_{\bl_{1,2}}\\+\sum_{a\in \cA}\< h_a^m\otimes h_a^m,\, \nabla^{(2)} P_tF(X^\cA)\>_{\bl_{1,2}\otimes \bl_{1,2}}.
  \end{multline*}
  By independence,
  \begin{multline*}
    \esp{\< X^\cA, \nabla P_tF(X^\cA)\>_{\bl_{1,2}}}\\
    = \sum_{a\in \cA} \esp{X_a\, \< h_a^m,\, \nabla P_tF(X^\cA)-\nabla
      P_tF(X^\cA_{\neg a}) \>_{\bl_{1,2}}}.
  \end{multline*}
  The fundamental theorem of calculus  now states that
  \begin{multline*}
   \esp{\< X^\cA, \nabla P_tF(X^\cA)\>_{\bl_{1,2}}}\\= \sum_{a\in \cA}\int_0^1 \esp{X_a^2\,
     \< h_a^m\otimes h_a^m,\, \nabla^{(2)} P_tF(X^\cA_{\neg a}+rX_a h_a^m\>_{\bl_{1,2}}}\dif r.
     \end{multline*}
     Since $\esp{X_a^2}=1$, the result follows by successive cancellations.
   \end{proof}

Introduce for any $a\prec b\in \cA$ and any $0\le s<t\le 1$,
\begin{multline*}
  \gc_{a,b}^m (s,t)\\=\left\{ \int_s^t \Bigl(h_{a_2}^m(r)-h_{a_2}^m(s)\Bigr)\dif
    h_{b_2}^m(r)- \int_s^t \Bigl(h_{b_2}^m(r)-h_{b_2}^m(s)\Bigr)\dif
    h_{a_2}^m(r)\right\} \ [e_{a_1},\, e_{b_1}].
\end{multline*}
\begin{lemma}
  \label{lem:majo_XaXb}
  Let $(U_a,\, a\in \cA)$ be a family of independent identically distributed
  random variables which belong to $L^p$. Then,
  \begin{equation*}
    \esp{     \left|  U_a\, \sum_{b\prec a}
        U_b\, \gc^m_{a,b} \right|_{\fs_{\eta,p}}^p }\le c \, m^{-(1/2-\eta)}\, \esp{|U_a|^p}^2.
  \end{equation*}
\end{lemma}
\begin{proof}
  By independence and as in the proof of Theorem~\ref{thm:majo_var},
  \begin{equation*}
    \esp{\left|  U_a\, \sum_{b\prec a}
        U_b\, \gc^m_{a,b}(s,t)\right|^{p}} \le c \, \esp{|U_a|^p}\esp{\left(\sum_{b\prec a} |U_b|^2 |\gc^m_{a,b}(s,t)|^2  \right)^{p/2}}.
  \end{equation*}
  Since $a$ is fixed and $b\prec a$, if $|t-s|\le 1/m$, $\gc_{a,b}^m$ is not
  zero only for $b_2=a_2-1$ and then, it is bounded by $m|t-s|^2$. If $|t-s|\ge
  1/m$, $\gc^m _{a,b}$ is not zero for at most $[dm|t-s|]$ values of $b$ and
  then, each term is bounded by $1/m$. Hence
  \begin{equation*}
    \esp{  \left( \sum_{b\prec a} |U_b|^2 |\gc^m_{a,b}(s,t)|^2\right)^{p/2} }\le c\,  m^{-p/2}(m|t-s|)^{p/2}\esp{|U_b|^p}.
  \end{equation*}
  The result follows by integration with respect to $\mu_{\eta,p}$.
\end{proof}

  \subsection{Proof of the main theorem}
\label{sec:main}
   The result of Lemma~\ref{lem:transformationL} raises a problem which did not
   exist in finite dimension: There is no apparent $m^{-1/2}$ factor which gives
   the rate of convergence after applying a Taylor expansion of the convenient order.
Said otherwise, there is no  clue that  the difference between $X^{\cA}$ and
$X^{\cA}_{\neg a}$ should be small. Actually, the $m^{-1/2}$ factor is hidden in
the $h_{a}^{m}$'s whose $L^{\infty}$ norm is exactly $m^{-1/2}$. But the
scalar product we have introduced involves their $\bl_{1,2}$ norm, which is $1$.
The necessary degree of freedom is given here by the possibility to consider
$h_{a}^{m}$ as an element of another functional space. We borrowed this idea from
\cite{shih_steins_2011}, our presentation being hopefully more straightforward. 
\begin{proof}[Proof of Theorem \protect{\ref{thm:final}}]
   For $t \ge 0$, $r\in [0,1]$, $y\in \fs_{\eta,p}$, let
  \begin{equation}\label{eq_donsker:4}
    X^\cA_{\neg a}(t , r, y)=e^{-t }(X^\cA_{\neg a}+rX_ah_a^{m})+\beta_{t}y.
  \end{equation}
  Lemma~\ref{lem:transformationL} and Proposition~\ref{prop:shih} imply that
  \begin{multline*}
    \esp{LP_tF(X^\cA)}\\
    \shoveleft{   = e^{-2t}\ \sum_{a\in \cA}\int_0^1\mathbf E\left[
      X_{a}^{2}\int_{\fs_{\eta,p}}\< \nabla^{(2)}F(X^\cA_{\neg a}\bigl(t , 0, y)\bigr)\right.\right.}\\
  \shoveright{\left.
      \left. -\nabla^{(2)}F\bigl(X^\cA_{\neg a}(t , r, y)\bigr),\,
        h_{a}^{m}\otimes h_{a}^{m} \>_{\bl_{1,2}^{\otimes 2}}\dif
        \P_{\eta,p}^{V}(y)\right]}\\
        \shoveleft{   = e^{-2t}\ \sum_{a\in \cA}\int_0^1\mathbf E\left[
      X_{a}^{2}\int_{\fs_{\eta,p}}\< \nabla^{(2)}F(X^\cA_{\neg a}\bigl(t , 1, y)\bigr)\right.\right.}\\
  \left.
      \left. -\nabla^{(2)}F\bigl(X^\cA_{\neg a}(t , 0, y)\bigr),\,
        h_{a}^{m}\otimes h_{a}^{m} \>_{\bl_{1,2}^{\otimes 2}}\dif \P_{\eta,p}^{V}(y)\right]
    \end{multline*}
    Since $\|h_{a}^{m}\|_{L^{2}}=m^{-1/2}$, for $F\in {\Sigma}_{\eta,p}$,
    \begin{multline}
      \label{eq_donsker:7}
      \left| \esp{LP_t\lF(X^\cA)} \right|\\ \le c\, e^{-2t}\sup_{\substack{r\in [0,1]\\ y,z\in \fs_{\eta,p}}} \esp{(1+|X_a|^2) \left\|X^\cA_{\neg a}(t , 0, y,
    z) -X^\cA_{\neg a}(t , r, y)\right\|_{\bfs_{\eta,p}}},
\end{multline}
where $a$ is any chosen index of $\cA$. By the definition of the norm on
$\bfs_{\eta,p}$, we have
    \begin{multline*}
\sup_{\substack{r\in [0,1]\\ y,z\in \fs_{\eta,p}}} \esp{(1+|X_a|^2) \left\|X^\cA_{\neg a}(t , 0, y,
    z) -X^\cA_{\neg a}(t , r, y)\right\|_{\bfs_{\eta,p}}}\\
\shoveleft{= \sup_{\substack{r\in [0,1]\\ y,z\in \fs_{\eta,p}}} \esp{(1+|X_a|^2) \left\|X^\cA_{\neg a}(t , 0, y,
    z) -X^\cA_{\neg a}(t , r, y)\right\|_{\fs_{\eta,p}}}}\\
\shoveleft{ +\sup_{\substack{r\in [0,1]\\ y,z\in \fs_{\eta,p}}} \mathbf E\Biggl[ (1+|X_a|^2) \biggl(  \iint \Bigl(         \pi_2\lS (X^\cA_{\neg a}(t , r, y)) } \\
       \shoveright{     -\pi_2\lS( X^\cA_{\neg a}(t , 0, y))
         \Bigr)_{u,v}^{p/2}\dif\mu_{\eta,p}(u,v) \biggr)^{1/p}\Biggr]}\\
         =A_{1}+A_{2}.
  \end{multline*}
  According to \eqref{eq_donsker:3} and \eqref{eq_donsker:4},
  \begin{equation}
    \label{eq_donsker:5}
    |A_{1}|\le e^{{-t}}\esp{(1+|X_{a}|^{2})|X_{a}| }\|h_{a}^{m}\|_{\fs_{\eta,p}}\le 2\,\esp{|X_{a}|^{3}}\, m^{-(1/2-\eta)}.
  \end{equation}

        Furthermore,
      \begin{multline*}
\Bigl(         \pi_2\lS (X^\cA_{\neg a}(t , r, y)) -\pi_2\lS (X^\cA_{\neg a}(t , 0, y)) \Bigr)_{u,v}\\
        =r\,X_a(t,y)\sum_{b\prec a}X_b(t,y)\, h_{a,b}^m(u,v),
      \end{multline*}
      where
      \begin{equation*}
        X_a(t,y)=e^{-t}X_a+\beta_{t}y. 
      \end{equation*}
      Hence, according to H\"older inequality,
      \begin{multline*}
        |A_{2}|\le c\,\esp{|X_{a}|^p}^{{2/p}}\\
        { \times \mathbf E\Biggr[\biggl(\iint\Bigl|  X_a(t,y)\sum_{b\prec a}X_b(t,y)\,
        h_{a,b}^m(u,v)}\Bigr|^{p/2} \dif\mu_{\eta,p}(u,v)
      \biggr)^{1/(p-2)}\Biggr]^{(p-2)/p}\\
      \shoveleft{ \le  c\,\esp{|X_{a}|^p}^{{2/p}}}\\
       {\times  \mathbf E\Biggr[\iint\Bigl|  X_a(t,y)\sum_{b\prec a}X_b(t,y)\,
        h_{a,b}^m(u,v)}\Bigr|^{p/2} \dif\mu_{\eta,p}(u,v)
      \Biggr]^{1/p}.
    \end{multline*}
    Lemma \ref{lem:majo_XaXb} implies that there exists $c>0$ such that
    for any $y\in \fs_{\eta,p}$,
    \begin{multline}\label{eq_donsker:6}
       |A_{2}|\le c\,\esp{|X_{a}|^p}^{{2/p}} \esp{|X_{a}|^{p/2}}^{2/p}\,
       m^{-(1/2-\eta)}\\ \le c\,\esp{|X_{a}|^p}^{{3/p}} \, m^{-(1/2-\eta)}.
     \end{multline}
     Since  $|\cA|=d.m$,
     plug \eqref{eq_donsker:5} and \eqref{eq_donsker:6} into \eqref{eq_donsker:7} to obtain the
     existence of $c>0$ such that for any $t>0$,
     \begin{equation}\label{eq_donsker:8}
        \Bigl|  \esp{LP_tF(X^\cA)}\Bigr|\le c\,e^{-2t}\esp{|X_{a}|^p}^{{3/p}} \, m^{-(1/2-\eta)}.
      \end{equation}
In view of \eqref{eq_donsker_wiener:2}, by integration over $\R^{+}$, we get
\eqref{eq_donsker:16} and the proof is complete. 
    \end{proof}

\providecommand{\bysame}{\leavevmode\hbox to3em{\hrulefill}\thinspace}
\providecommand{\MR}{\relax\ifhmode\unskip\space\fi MR }
\providecommand{\MRhref}[2]{%
  \href{http://www.ams.org/mathscinet-getitem?mr=#1}{#2}
}
\providecommand{\href}[2]{#2}


\end{document}